\documentclass[12pt,reqno]{amsart}
\usepackage[margin=1in]{geometry}
\usepackage{amssymb,amsfonts,amsmath,mathrsfs,etoolbox,mathtools,bm,bbm}
\usepackage[UKenglish]{isodate}
\mathtoolsset{showonlyrefs}
\usepackage[dvipsnames]{xcolor}
\definecolor{default}{rgb}{.26,.39,.77}
\usepackage[breaklinks=true,linktocpage=true,colorlinks=true,linkcolor=default,citecolor=default,filecolor=default,urlcolor=black!39!blue,backref=page]{hyperref}
\usepackage{orcidlink}
\allowdisplaybreaks

\newtheorem{theorem}{Theorem}[section]
\newtheorem{corollary}[theorem]{Corollary}
\newtheorem{definition}[theorem]{Definition}
\newtheorem{proposition}[theorem]{Proposition}

\theoremstyle{definition}
\newtheorem{remark}[theorem]{Remark}

\numberwithin{equation}{section}

\renewcommand{\Re}{\mathrm{Re}}
\renewcommand{\Im}{\mathrm{Im}}
\renewcommand{\epsilon}{\varepsilon}
\DeclareMathOperator*{\vol}{vol}
\patchcmd{\section}{\scshape}{\bfseries}{}{}
\makeatletter
\renewcommand{\@secnumfont}{\bfseries}
\makeatother
\newcommand{\norm}[1]{\left\lVert#1\right\rVert}
\newcommand{\MYhref}[3][black!39!blue]{\href{#2}{\color{#1}{#3}}}
\renewcommand*\backref[1]{$\uparrow$\thinspace\ifx#1\relax \else #1 \fi}

\begin{document}

\title[Mean Values and Quantum Variance in Higher Rank]{Mean Values and Quantum Variance for Degenerate Eisenstein Series of Higher Rank}

\author[Dimitrios Chatzakos]{Dimitrios Chatzakos\,\orcidlink{0000-0003-2124-5552}}
\address[Dimitrios Chatzakos]{Department of Mathematics, School of Natural Sciences, University of Patras, 265 04 Rion, Patras, Greece}
\urladdr{\MYhref{https://dchatzakos.math.upatras.gr}{https://dchatzakos.math.upatras.gr}}
\email{dchatzakos@math.upatras.gr}

\author[Corentin Darreye]{Corentin Darreye\,\orcidlink{0000-0002-7421-8147}}
\address[Corentin Darreye]{}
\urladdr{\MYhref{https://www.mathgenealogy.org/id.php?id=271120}{https://www.mathgenealogy.org/id.php?id=271120}}
\email{corentin.darreye@orange.fr}

\author[Ikuya Kaneko]{Ikuya Kaneko\,\orcidlink{0000-0003-4518-1805}}
\address[Ikuya Kaneko]{The Division of Physics, Mathematics and Astronomy, California Institute of Technology, 1200 East California Boulevard, Pasadena, California 91125, United States of America}
\urladdr{\MYhref{https://sites.google.com/view/ikuyakaneko/}{https://sites.google.com/view/ikuyakaneko/}}
\email{ikuyak@icloud.com}

\thanks{DC acknowledges the financial support from the IdEx (Initiative d'Excellence) Postdoctoral Fellowship at the Institut de Math\'{e}matiques de Bordeaux, and from ELKE of the University of Patras (MEDIKOS~Program No. $82043$). IK expresses gratitude for the diverse support from the Masason Foundation from 2019 to 2024.}

\subjclass[2020]{Primary \MYhref[black]{https://mathscinet.ams.org/mathscinet/msc/msc2020.html?t=&s=11f12}{11F12}, \MYhref[black]{https://mathscinet.ams.org/mathscinet/msc/msc2020.html?t=&s=11f72}{11F72}; Secondary \MYhref[black]{https://mathscinet.ams.org/mathscinet/msc/msc2020.html?t=&s=58j51}{58J51}, \MYhref[black]{https://mathscinet.ams.org/mathscinet/msc/msc2020.html?t=&s=81q50}{81Q50}}


\cleanlookdateon

\date{\today}

\begin{abstract}
We investigate the mean value of the inner product of squared $\mathrm{GL}_{n}$ degenerate maximal parabolic Eisenstein series against a smooth compactly supported function lying in a restricted space of incomplete Eisenstein series induced from a $\mathrm{SL}_{2}(\mathbb{Z})$ Hecke--Maa{\ss} cusp form $\varphi$. Our result breaks the fundamental threshold with a polynomial power-saving beyond the pointwise implications of the generalised Lindel\"{o}f hypothesis for $L$-functions attached to $\varphi$. Furthermore, we evaluate the archimedean quantum variance and establish approximate orthogonality, expanding upon Zhang's (2019) work on quantum unique ergodicity for $\mathrm{GL}_{n}$ degenerate maximal parabolic Eisenstein series as well as Huang's (2021) work on quantum variance for $\mathrm{GL}_{2}$ Eisenstein series. Despite the theoretical strength of these manifestations, our argument relies exclusively on the Watson--Ichino-type formula for incomplete Eisenstein series of type $(2, 1, \ldots, 1)$ and Jutila's (1996) asymptotic formula for the second moment of $L$-functions attached to $\varphi$ in long intervals, supplemented by a standard analytical toolbox.
\end{abstract}

\maketitle
\tableofcontents

\section{Introduction}

\subsection{Historical prelude}
Quantum ergodicity represents a foundational topic in the theory of automorphic forms. Its origins pivot around the realm of mathematical physics where, for a smooth compact Riemannian manifold $\mathcal{M}$, one addresses the behaviour of eigenfunctions of the Laplace--Beltrami operator on $\mathcal{M}$ as the eigenvalue tends to infinity. This type of question pertains to the dynamics of the ergodic flow. Landmark results of \v{S}nirel'man~\cite{Shnirelman1974,Shnirelman1993},\footnote{As Zelditch notes~\cite[Page~921]{Zelditch1994}, a more rigorous paper by \v{S}nirel'man was located by Michael~Tabor at Columbia University, and a translation from the Russian original by Semyon Dyatlov is now available~\cite{Shnirelman2022}. The interested reader is directed to~\cite{Shnirelman2023} for the historical background.} Colin de Verdi\`{e}re~\cite{ColindeVerdiere1985}, and Zelditch~\cite{Zelditch1987} assert in general that if the geodesic flow on the unit cotangent bundle is ergodic, then there exists a density one subsequence of Laplace eigenfunctions whose mass equidistributes on $\mathcal{M}$ as the eigenvalue tends to infinity. Rudnick and Sarnak~\cite[Page~196]{RudnickSarnak1994} predicted that when $\mathcal{M}$ is hyperbolic (i.e. of negative sectional curvature), this property holds for the entire sequence of Laplace eigenfunctions, namely the probability measures $|\varphi|^{2} d\mu$ converge weakly to the normalised measure $\frac{1}{\vol(\mathcal{M})} d\mu$, where $d\mu$ is the volume element on $\mathcal{M}$. This statement is referred to as the quantum unique ergodicity (QUE) conjecture; see~\cite{Dyatlov2022,Zelditch1992,Zelditch2006,Zelditch2010,Zelditch2019} for surveys of quantum ergodicity, and~\cite{BisainHumphriesMandelshtamWalshWang2024,Marklof2006,Sarnak1995,Sarnak2003,Sarnak2011} for surveys of arithmetic quantum chaos.

A parallel line of inquiry extends to higher rank homogeneous spaces $\mathbb{X} = \Gamma \backslash G/K$, where new perspectives come into the picture. If $\mathbb{X}$ has rank $r \geq 2$, then the centre of the universal enveloping algebra of $\mathbb{X}$ is generated by $r$ distinct Casimir operators. For $n \geq 2$, we define
\begin{equation*}
\mathbb{X}_{n} \coloneqq \mathrm{SL}_{n}({\mathbb{Z}}) \backslash \mathbb{H}^{n} \cong \mathrm{Z}(\mathbb{R}) \, \mathrm{SL}_{n}({\mathbb{Z}}) \backslash \mathrm{GL}_{n}({\mathbb{R}}) / \mathrm{O}_{n}({\mathbb{R}}),
\end{equation*}
where $\mathbb{H}^{n}$ denotes the standard generalised upper half-plane. This space has rank $r = n-1$, wherein automorphic forms are eigenfunctions of all the Casimir operators. Moreover, $\mathbb{X}_{n}$ is also equipped with the continuous spectrum spanned by Langlands Eisenstein series.

For $n = 2$, the problem was first investigated in the seminal work of Luo and Sarnak~\cite[Theorem~1.1]{LuoSarnak1995}, which proved a modified version of QUE for Eisenstein series on the modular orbifold $\mathbb{X}_{2}$. Furthermore, Watson~\cite[Theorem~4]{Watson2002} demonstrated that effective QUE for Hecke--Maa{\ss} cusp forms follows from subconvex bounds for triple product $L$-functions. The combined work of Lindenstrauss~\cite[Theorem~1.4]{Lindenstrauss2006} and Soundararajan~\cite[Theorem~1]{Soundararajan2010-2} resolves the QUE conjecture on $\mathbb{X}_{2}$, albeit with an ineffective rate of equidistribution.

For $n \geq 3$, QUE for automorphic forms on compact quotients of $\mathrm{SL}_{n}({\mathbb R})$ was first examined by Silberman and Venkatesh~\cite{Silberman2015,SilbermanVenkatesh2007,SilbermanVenkatesh2019} using ergodic methods. In the noncompact setting, Zhang~\cite{Zhang2017,Zhang2019} generalised the result of Luo and Sarnak~\cite[Theorem~1.1]{LuoSarnak1995} to $n \geq 3$, thereby establishing a modified version of QUE for degenerate maximal parabolic (unitary) Eisenstein series $E_{(n-1, 1)}(z, \frac{1}{2}+it, \mathbbm{1})$ induced from the maximal parabolic subgroup $\mathrm{P}_{(n-1, 1)}(\mathbb{Z})$ and the constant eigenfunction $\varphi_{0} = \mathbbm{1}$. For any $n \geq 2$, any fixed test function $f \in \mathcal{C}_{c}^{\infty}(\mathbb{X}_{n})$, and any $t \in \mathbb{R}$, we denote the $n$-dimensional Petersson inner product by
\begin{equation}\label{eq:inner-product}
\mu_{n, t}(f) \coloneqq \left\langle f, \left|E_{(n-1, 1)} \left(\cdot, \frac{1}{2}+it, \mathbbm{1} \right) \right|^{2} \right\rangle_{n} 
= \int_{\mathbb{X}_{n}} f(z) \left|E_{(n-1, 1)} \left(z, \frac{1}{2}+it, \mathbbm{1} \right) \right|^{2} d\mu_{n}(z),
\end{equation}
where $d\mu_{n}(z)$ denotes the standard Haar measure. Zhang~\cite[Theorem~1.3]{Zhang2019} then proved the asymptotic formula
\begin{equation}\label{eq:Zhang}
\mu_{n, t}(f) \sim \frac{2 \log t}{\Lambda(n)} \langle f, 1 \rangle_{n}
\end{equation}
as $t \to \infty$, where $\Lambda(s) \coloneqq \pi^{-\frac{s}{2}} \Gamma(\frac{s}{2}) \zeta(s)$ denotes the completed Riemann zeta function.\footnote{This notation ought not to be confused with the von Mangoldt function, particularly when its argument is a positive integer $n \in \mathbb{N}$.}~The proof is predicated upon the spectral expansion and the estimation of $\mu_{n, t}$ for all automorphic forms in $L^{2}(\mathbb{X}_{n})$. If $n = 2$, then $f$ decomposes into Hecke--Maa{\ss} cusp forms $\varphi$ with spectral parameter $t_{\varphi} > 0$, the unitary Eisenstein series $E(z, \frac{1}{2}+it)$, and the constant eigenfunction $\varphi_{0} = \mathbbm{1}$. To verify~\eqref{eq:Zhang} for $n = 2$, Luo and Sarnak~\cite[Proposition~2.1]{LuoSarnak1995} demonstrated that the contribution of the discrete spectrum is negligibly small, namely $\mu_{2, t}(\varphi) \ll_{\varphi, \epsilon} (1+|t|)^{-\frac{1}{6}+\epsilon}$, whereas the assumption of the generalised Lindel\"{o}f hypothesis implies of improved exponent of $-\frac{1}{2}+\epsilon$. Furthermore, the mean value result due to Huang~\cite[Theorem~1]{Huang2021-3} asserts~that if $\varphi$ is an even Hecke--Maa{\ss} cusp form, then
\begin{equation*}
\frac{1}{T} \int_{T}^{2T} \mu_{2, t}(\varphi) dt = o_{\varphi}(T^{-\frac{1}{2}}).
\end{equation*}
Note that $\mu_{2, t}(\varphi) = 0$ identically if $\varphi$ is odd. This surpasses the barrier beyond the pointwise implications of the generalised Lindel\"{o}f hypothesis, albeit with an infinitesimal improvement. Given any two Hecke--Maa{\ss} cusp forms $\varphi$ and $\psi$, we define the quantum variance by\footnote{This definition aligns with that of Huang~\cite[Page~1226]{Huang2021-3}.}
\begin{equation}\label{eq:Q-2}
Q_{2}(\varphi, \psi) \coloneqq \lim_{T \to \infty} \frac{1}{\log T} \int_{T}^{2T} \mu_{2, t}(\varphi) \overline{\mu_{2, t}(\psi)} dt 
\end{equation}
Then the second result of Huang~\cite[Theorem~2]{Huang2021-3} asserts that
\begin{equation*}
Q_{2}(\varphi, \psi) = 
\begin{dcases}
C_{2}(\varphi) L \left(\frac{1}{2}, \varphi \right)^{2} V_{2}(\varphi) & \text{if $\varphi = \psi$ is even},\\
0 & \text{otherwise},
\end{dcases}
\end{equation*}
where $C_{2}(\varphi)$ is given by a product of local densities, $L(s, \varphi)$ denotes the standard $L$-function attached to $\varphi$, and $V_{2}(\varphi)$ specialises to the diagonal case $\varphi = \psi$ of a nonnegative Hermitian form $V_{2}(\varphi, \psi)$ defined on $\mathcal{C}_{c}^{\infty}(\mathbb{X}_{2})$, namely
\begin{equation}\label{eq:V-2}
V_{2}(\varphi) \coloneqq \frac{|\Gamma(\frac{1}{4}+\frac{it_{\varphi}}{2})|^{4}}{2\pi|\Gamma(\frac{1}{2}+it_{\varphi})|^{2}}.
\end{equation}
Although no currently available ergodic technique renders the rate of decay in QUE effective, Luo and Sarnak~\cite[Theorem~1.2]{LuoSarnak1995} considered the quantum variance and demonstrated an optimal rate of convergence on average, which can be further refined for test functions $f$ in the subspace $\mathcal{C}_{c, 0}^{\infty}(\mathbb{X}_{2})$ of $\mathcal{C}_{c}^{\infty}(\mathbb{X}_{2})$ consisting of all functions with zero mean value $\langle f, 1 \rangle_{2} = 0$~and vanishing zeroth Fourier coefficient $\int_{0}^{1} f(x+iy) dx = 0$ for $y$ sufficiently large. Zhao~\cite{Zhao2009,Zhao2010} addressed this problem via the Kuznetsov formula, thereby extending the prior result of Luo and Sarnak~\cite[Theorem~1]{LuoSarnak2004} in the holomorphic case. In the nonholomorphic case, the subsequent work of Sarnak and Zhao~\cite[Proposition~7]{SarnakZhao2019} achieves the stronger result
\begin{equation*}
\lim_{T \to \infty} \frac{1}{T} \sum_{t_{\psi} \leq T} |\langle \varphi, |\psi|^{2} \rangle_{2}|^{2} = C_{2}^{\star}(\varphi) L \left(\frac{1}{2}, \varphi \right) V_{2}(\varphi), \qquad C_{2}^{\star}(\varphi) \coloneqq \frac{1}{\zeta(2)} \prod_{p} \left(1 -\frac{\lambda_{\varphi}(p)}{p^{\frac{3}{2}}+p^{\frac{1}{2}}} \right),
\end{equation*}
where $\lambda_{\varphi}(p)$ denotes the $p$-th Hecke eigenvalue of $\varphi$. Among other notable applications, the nonnegativity of the central $L$-values $L(\frac{1}{2}, \varphi) \geq 0$ follows as a straightforward consequence.

Furthermore, the quantum variance for Hecke--Maa{\ss} cusp forms $\varphi$ is in principle consistent with a prediction by Feingold and Peres~\cite{EckhardtFishmanKeatingAgamMainMuller1995,FeingoldPeres1986} on the quantum variance for generic chaotic systems. In particular, it recovers the classical variance due to Ratner~\cite{Ratner1973,Ratner1987}, up to multiplication by $L(\frac{1}{2}, \varphi)$ for the discrete spectrum and by $L(\frac{1}{2}, \varphi)^{2}$ for the continuous spectrum, with the difference arising from the distinct shape of the Watson--Ichino formul{\ae}.

Utilising the theta correspondence, Nelson~\cite{Nelson2016,Nelson2017,Nelson2019-2} investigated the quantum variance for compact arithmetic quotients, thereby reducing the problem to the estimation of metaplectic Rankin--Selberg convolutions. The cofinite case differs from the cocompact case, since the latter has no cusps and hence lacks an adequate Hecke theory for Fourier coefficients of Hecke--Maa{\ss} cusp forms. It can be summarised that that the analysis of quantum variance for arithmetic orbifolds sheds some light on QUE and related subjects, wherein an additional layer of arithmetical information -- the Kuznetsov and Watson--Ichino formul{\ae} -- enters. For further advancement concerning quantum variance from a number-theoretic~perspective, the reader is directed to~\cite{ChatzakosFrotRaulf2021,HuangLester2023,SunZhang2024,Zenz2021,Zenz2022}.

\subsection{Statement of results}
The approach of Zhang~\cite{Zhang2017,Zhang2019} reduces the analysis to determining the size of the inner products $\mu_{n, t}(g)$ in~\eqref{eq:inner-product} for a Hecke--Maa{\ss} cusp form~$g = \varphi$ and for a parabolic Eisenstein series $g = E_{(n_{1}, n_{2}, \ldots, n_{r})}$ of type $(n_{1}, n_{2}, \ldots, n_{r})$ associated to~the parabolic subgroup $\mathrm{P}_{(n_{1}, n_{2}, \ldots, n_{r})}$ with $n_{1}+n_{2}+\cdots+n_{r} = n$. For $n \geq 3$, crucial distinctions arise compared to the case of $n = 2$, notably the fact that $\mu_{n, t}$ vanishes identically on Hecke--Maa{\ss} cusp forms and the majority of Eisenstein series; see~\cite[Section~4]{Zhang2019}. Consequently, the analysis of QUE is centered on~\eqref{eq:inner-product} in a certain subfamily of Eisenstein series, identified therein as including only that of type $(2, 1, \ldots, 1)$ and the minimal parabolic Eisenstein series of type $(1, \ldots, 1)$. Given the absence of square-integrability of Langlands Eisenstein series on the underlying $n$-dimensional orbifold $\mathbb{X}_{n}$, the concept of incomplete Eisenstein series comes into play.

In harmony with the case of $n = 2$, the contribution of the minimal parabolic Eisenstein series of type $(1, \ldots, 1)$ constitutes the main term as in~\eqref{eq:Zhang}, whereas the contribution of the incomplete Eisenstein series of type $(2, 1, \ldots, 1)$ remains small. Let $E_{(2, 1, \ldots, 1)}(z, \eta, \varphi)$ denote such an incomplete Eisenstein series induced from a $\mathrm{SL}_{2}(\mathbb{Z})$ Hecke--Maa{\ss} cusp form $\varphi$, where $\eta$ is a fixed smooth compactly supported function. Then Zhang~\cite[Theorem~4.1]{Zhang2019} proved
\begin{equation*}
\mu_{n, t}(E_{(2, 1, \ldots, 1)}(\cdot, \eta, \varphi)) \ll_{n, \varphi, \epsilon} (1+|t|)^{-\frac{1}{2}+\epsilon}.
\end{equation*}
In principle, his machinery yields the exponent $-\frac{n}{2}+1+\epsilon$, whereas the generalised Lindel\"{o}f hypothesis for $L(s, \varphi)$ implies an improved upper bound
\begin{equation}\label{eq:GLH}
\mu_{n, t}(E_{(2, 1, \ldots, 1)}(\cdot, \eta, \varphi)) \ll_{n, \varphi, \epsilon} (1+|t|)^{-\frac{n-1}{2}+\epsilon}.
\end{equation}
Our first main result gives a mean value estimate for the inner product against the incomplete Eisenstein series of type $(2, 1, \ldots, 1)$. Contrary to heuristic expectations, it not only supports the conjecture~\eqref{eq:GLH}, but also achieves a polynomial power-saving improvement beyond the pointwise implications of the generalised Lindel\"{o}f hypothesis, indicating among other aspects substantial fluctuations of the inner product within a dyadic interval $t \in [T, 2T]$.
\begin{theorem}\label{thm:main}
Keep the notation as above. Then we have for any $n \geq 3$ that
\begin{equation*}
 \frac{1}{T} \int_{T}^{2T} \mu_{n, t}(E_{(2, 1, \ldots, 1)}(\cdot, \eta, \varphi)) dt \ll_{n, \varphi} T^{-\frac{n-1}{2}-\frac{7(n-2)}{6(5n-2)}} \sqrt{\log T}.
\end{equation*}
\end{theorem}

Theorem~\ref{thm:main} exhibits an additional polynomial power-saving that obeys
\begin{equation*}
\frac{7}{78} = 0.08974 \cdots \leq \frac{7(n-2)}{6(5n-2)} < \frac{7}{30} = 0.23333 \cdots.
\end{equation*}
Given $n \geq 3$ and any two Hecke--Maa{\ss} cusp forms $\varphi$ and $\psi$, we define the quantum variance in higher rank by
\begin{equation*}
Q_{n}(\varphi, \psi) \coloneqq \lim_{T \to \infty} \frac{T^{n-2}}{\log T} \int_{T}^{2T} \mu_{n, t}(E_{2, 1, \ldots, 1}(\cdot, \eta, \varphi)) \overline{\mu_{n, t}(E_{2, 1, \ldots, 1}(\cdot, \eta, \psi))} dt,
\end{equation*}
which extends~\eqref{eq:Q-2} for $n = 2$. The second main result of the present paper is the following.
\begin{theorem}\label{thm:quantum-variance}
Keep the notation as above. Then we have for any $n \geq 3$ that
\begin{equation*}
Q_{n}(\varphi, \psi) = 
\begin{dcases}
C_{n}(\varphi) L \left(\frac{3-n}{2}, \varphi \right)^{2} V_{n}(\varphi) & \text{if $\varphi = \psi$ is even},\\
0 & \text{otherwise},
\end{dcases}
\end{equation*}
where $C_{n}(\varphi)$ is given by a product of local densities~\eqref{eq:C-n}, and
\begin{equation}\label{eq:V}
V_{n}(\varphi) \coloneqq \frac{|\Gamma(\frac{3-n}{4}+\frac{it_{\varphi}}{2})|^{4}}{2\pi^{3-n} |\Gamma(\frac{1}{2}+it_{\varphi})|^{2}}.
\end{equation}
\end{theorem}

Our results are commensurate with prior investigations for $n = 2$, where a $\mathrm{SL}_{2}(\mathbb{Z})$ Hecke--Maa{\ss} cusp form $\varphi$ is replaced by the Eisenstein series of type $(2, 1, \ldots, 1)$ induced from $\varphi$, thereby establishing the folklore conjecture~\eqref{eq:GLH} on average for any $n \geq 3$.

\begin{remark}
A salient difference between the results of Luo and Sarnak~\cite[Theorem~1.1]{LuoSarnak1995} and Zhang~\cite[Theorem~1.3]{Zhang2019} lies in the fact that the former asserts for $n = 2$ that
\begin{equation*}
\mu_{2, t}(f)-\frac{2\log t}{\Lambda(2)} \langle f, 1 \rangle_{2} \sim \Re \left(\frac{\zeta^{\prime}}{\zeta}(1+2it) \right) \neq O(1)
\end{equation*}
for $f \in \mathcal{C}_{c}^{\infty}(\mathbb{X}_{2})$, with the error term coming from the continuous spectrum, whereas the latter asserts for $n \geq 3$ that the contribution of the incomplete Eisenstein series is bounded,~namely
\begin{equation}\label{eq:bounded}
\mu_{n, t}(f)-\frac{2\log t}{\Lambda(n)} \langle f, 1 \rangle_{n} = O(1)
\end{equation}
for $f \in \mathcal{C}_{c}^{\infty}(\mathbb{X}_{n})$. It would be a formidable task to derive an elegant expression for the second moment of the discrepancy on the left-hand side of~\eqref{eq:bounded} for $f$ within a certain restricted~space of smooth compactly supported functions. While it is likely feasible given the state-of-the-art techniques, we shall defer addressing this question to subsequent investigations.
\end{remark}

\begin{remark}
Quantum ergodicity in higher rank was studied by Brumley and Matz~\cite[Theorem~1.1]{BrumleyMatz2023} in the level aspect, interpreted through the Benjamini--Schramm limit, thereby extending the foundational work of Le Masson and Sahlsten~\cite[Theorem~1.1]{LeMassonSahlsten2017} and~Abert, Bergeron, and Le Masson~\cite[Theorem~5]{AbertBergeronLeMasson2024} in the rank $1$ setting.
\end{remark}

\subsection{Proof strategies}
We succinctly outline aspects of our methods. A crucial ingredient of the proof of Theorem~\ref{thm:main} is an approximation argument for smooth compactly~supported functions on $\mathbb{H}^{n}$ (see~\cite[Remark~2.2]{Zhang2019}), which is in harmony with~\cite[Proposition~2.3]{LuoSarnak1995}.

A second key observation pertains to the absence of the Watson--Ichino formula in higher rank. As a substitute, we utilise the vanishing of the measure $\mu_{n, t}$ on Hecke--Maa{\ss} cusp~forms for $n \geq 3$ and on Eisenstein series of types other than $(2, 1, \ldots, 1)$. This property is specific to the higher rank context and lies in the spirit of Langlands's theory of constant terms~\cite{Langlands1971}. Then a subtle rearrangement of the computation of Zhang~\cite[Proposition~4.1]{Zhang2019} serves as a proper alternative to the Watson--Ichino formula, thereby relating the problem to a certain weighted integral moment of $\mathrm{GL}_{2}$ $L$-functions. To establish the polynomial power-saving for the inner product in question, we make use of the Cauchy--Schwarz inequality, approximate functional equations, and the stationary phase analysis, followed by delicate manipulations of the optimisation problem that emerge from these combined techniques.

The proof of Theorem~\ref{thm:quantum-variance} relies similarly on the Watson--Ichino-type formula for parabolic Eisenstein series in higher rank, thereby leading to the problem of determining the behaviour of the second moment of completed $\mathrm{GL}_{2}$ $L$-functions of the shape
\begin{equation}\label{eq:shifted-second-moments}
\frac{1}{T} \int_{T}^{2T} \Lambda \left(\frac{1}{2}+it, \varphi \right) \Lambda \left(\frac{1}{2}-it, \psi \right) dt,
\end{equation}
where $\varphi \neq \psi$. This problem has not garnered adequate attention. In particular, the analysis of~\eqref{eq:shifted-second-moments} requires a meticulous treatment of $L(s, \varphi)$ and some oscillatory factors arising from the archimedean component $L_{\infty}(s, \varphi)$. This is where the best known second moment estimate of Jutila~\cite[Theorem~3]{Jutila1996} -- a nonholomorphic counterpart of the result of Good~\cite[Theorem]{Good1982} -- comes into play. Note that Huang~\cite[Page~1233]{Huang2021-3} employed a suboptimal asymptotic formula of Kuznetsov~\cite[Theorem~1]{Kuznetsov1981}, highlighting the main term; however, the error term dictates the extent of the polynomial power-saving in Theorem~\ref{thm:main} when~$n \geq 3$.

\section{Automorphic preliminaries}
This section compiles basic ingredients from the spectral theory of automorphic forms on $\mathbb{X}_{n} = \mathrm{Z}(\mathbb{R}) \, \mathrm{SL}_{n}({\mathbb{Z}}) \backslash \mathrm{GL}_{n}({\mathbb{R}}) / \mathrm{O}_{n}({\mathbb{R}})$. For further details, the reader is directed to the book of Goldfeld~\cite{Goldfeld2006}.

\subsection{Iwasawa decomposition}
Recall that the generalised upper half-plane $\mathbb{H}^{n}$ is defined as the quotient $\mathrm{GL}_{n}(\mathbb{R})/(\mathrm{O}_{n}(\mathbb{R}) \cdot \mathrm{Z}(\mathbb{R}))$ and may be identified~\cite[Definition~1.2.3]{Goldfeld2006} with the set of all $n \times n$ matrices of the form $z = x \cdot y$, where
\begin{equation}\label{eq:Iwasawa}
x = \begin{pmatrix} 1 & x_{1, 2} & x_{1, 3} & \cdots & x_{1, n-1} & x_{1, n} \\ & 1 & x_{2, 3} & \cdots & x_{2, n-1} & x_{2, n}\\ & & \ddots & & & \vdots \\ & & & & 1 & x_{n-1, n}\\ & & & & & 1 \end{pmatrix}, \qquad 
y = \begin{pmatrix} y_{1} \cdots y_{n-1} & & & & \\ & y_{1} \cdots y_{n-2} & & & \\ & & \ddots & & \\ & & & y_{1} & \\ & & & & 1 \end{pmatrix}
\end{equation}
with $x_{i, j} \in \mathbb{R}$ and $y_{i} > 0$, known as the Iwasawa decomposition for $\mathrm{GL}_{n}(\mathbb{R})$. The generalised upper half-plane $\mathbb{H}^{n}$ is $\frac{(n+2)(n-1)}{2}$-dimensional over $\mathbb{R}$ and does not necessarily have a complex structure. Furthermore, the associated Haar measure is specified by
\begin{equation*}
d\mu_{n}(z) \coloneqq \prod_{1 \leq i < j \leq n} x_{i, j} \prod_{k = 1}^{n-1} y_{k}^{-k(n-k)} \frac{dy_{k}}{y_{k}}.
\end{equation*}
We adopt the standard convention of Goldfeld~\cite[Theorem~1.6.1]{Goldfeld2006}, namely
\begin{equation*}
\vol(\mathbb{X}_{n}) = n \prod_{k = 2}^{n} \Lambda(k),
\end{equation*}
whereas Zhang~\cite[Page~5]{Zhang2019} normalises the Haar measure such that $\vol(\mathbb{X}_{n}) = 1$.

\subsection{Hecke--Maa{\ss} cusp forms}
We define Hecke--Maa{\ss} cusp forms attached to $\mathrm{SL}_{n}(\mathbb{Z})$ in anticipation of their pivotal role in the subsequent formulation of parabolic Eisenstein series. If $\mathfrak{D}_{n}$ denotes the centre of the universal enveloping algebra $\mathfrak{gl}_{n}(\mathbb{R})$, then there exists a family of differential operators (the Casimir operators) $\Delta_{m, n}$ parametrised by $n \neq 2$ and $2 \leq m \leq n$, which generate $\mathfrak{D}_{n}$ as a polynomial algebra of rank $n-1$; see~\cite[Section~2.3]{Goldfeld2006}.

Given $\nu = (\nu_{1}, \ldots, \nu_{n-1}) \in \mathbb{C}^{n-1}$, we define
\begin{equation*}
I_{\nu}(z) \coloneqq \prod_{i = 1}^{n-1} \prod_{j = 1}^{n-1} y_{i}^{b_{i, j} \nu_{j}}, \qquad b_{i, j} \coloneqq 
\begin{cases}
ij & \text{if $i+j \leq n$},\\
(n-i)(n-j) & \text{otherwise},
\end{cases}
\end{equation*}
which are eigenfunctions of all the Casimir operators $\Delta_{m, n}$ with eigenvalues $\lambda_{m, n}$, namely
\begin{equation*}
\Delta_{m, n} I_{\nu}(z) = \lambda_{m, n} I_{\nu}(z).
\end{equation*}
An automorphic form $\varphi(z)$ of type $\nu$ in $L^{2}(\mathbb{X}_{n})$ is a smooth function satisfying the invariance $\varphi(\gamma z) = \varphi(z)$ for $\gamma \in \mathrm{SL}_{n}(\mathbb{Z})$ and the condition
\begin{equation*}
\Delta_{m, n} \varphi(z) = \lambda_{m, n} \varphi(z).
\end{equation*}
Let $\sum_{i = 1}^{m} r_{i} = n$. If the additional cuspidality condition
\begin{equation*}
\int_{(\mathrm{SL}_{n}(\mathbb{Z}) \cap \mathrm{U}) \backslash \mathrm{U}} \varphi(uz) du = 0
\end{equation*}
holds for all matrix groups of the shape
\begin{equation*}
\mathrm{U} \coloneqq \left\{\begin{pmatrix} I_{r_{1}} & \ast & \cdots & \ast \\ 0 & I_{r_{2}} & \cdots & \ast \\ \vdots & \vdots & \ddots & \vdots \\ 0 & 0 & \cdots & I_{r_m} \end{pmatrix} \right\} \subset \mathrm{GL}_{n}(\mathbb{R}),
\end{equation*}
then $\varphi$ is called a Maa{\ss} cusp form. We assume without loss of generality that all Maa{\ss} cusp forms are joint eigenfunctions of every Hecke operator, unless otherwise specified; see~\cite[Section~9.3]{Goldfeld2006} for~the Hecke theory in higher rank.

\subsection{Parabolic subgroups}
There exist various versions of Eisenstein series of higher rank, contingent upon the parabolic subgroups of $\mathrm{GL}_{n}(\mathbb{R})$. The parabolic subgroup $\mathrm{P}_{(n_{1}, n_{2}, \ldots, n_{r})}(\mathbb{R})$ associated to the partition $n_{1}+n_{2}+\cdots+n_{r}=n$ is defined by
\begin{equation*}
\mathrm{P}_{(n_{1}, n_{2}, \ldots, n_{r})}(\mathbb{R}) = \left\{\begin{pmatrix} \mathfrak{m}_{n_{1}} & \ast & \cdots & \ast \\ 0 & \mathfrak{m}_{n_{2}} & \cdots & \ast \\ \vdots & \vdots & \ddots & \vdots \\ 0 & 0 & \cdots & \mathfrak{m}_{n_{r}} \end{pmatrix} \right\} \subset \mathrm{GL}_{n}(\mathbb{R}),
\end{equation*}
where $\mathfrak{m}_{n_{i}} \in \mathrm{GL}_{n_{i}}(\mathbb{R})$. Two distinct parabolic subgroups defined by the partitions $(n_{1}, \ldots, n_{r})$ and $(n_{1}^{\prime}, \ldots, n_{r}^{\prime})$, respectively, are termed associate if the set $\left\{n_{1}, \ldots, n_{r} \right\}$ forms a permutation of $\{n_{1}^{\prime}, \ldots, n_{r}^{\prime} \}$. The parabolic subgroup admits a Langlands decomposition
\begin{equation*}
\mathrm{P}_{(n_{1}, n_{2}, \ldots, n_{r})}(\mathbb{R}) = \mathrm{N}_{(n_{1}, n_{2}, \ldots, n_{r})}(\mathbb{R}) \cdot \mathrm{M}_{(n_{1}, n_{2}, \ldots, n_{r})}(\mathbb{R}),
\end{equation*}
where 
\begin{equation*}
\mathrm{N}_{(n_{1}, n_{2}, \ldots, n_{r})}(\mathbb{R}) \coloneqq \left\{\begin{pmatrix} I_{n_{1}} & \ast & \cdots & \ast \\ 0 & I_{n_{2}} & \cdots & \ast \\ \vdots & \vdots & \ddots & \vdots \\ 0 & 0 & \cdots & I_{n_{r}} \end{pmatrix} \right\} \subset \mathrm{GL}_{n}(\mathbb{R})
\end{equation*}
denotes the unipotent radical, and
\begin{equation*}
\mathrm{M}_{(n_{1}, n_{2}, \ldots, n_{r})}(\mathbb{R}) \coloneqq \left\{\begin{pmatrix} \mathfrak{m}_{n_{1}} & 0 & \cdots & 0 \\ 0 & \mathfrak{m}_{n_{2}} & \cdots & 0 \\ \vdots & \vdots & \ddots & \vdots \\ 0 & 0 & \cdots & \mathfrak{m}_{n_{r}} \end{pmatrix} \right\} \subset \mathrm{GL}_{n}(\mathbb{R})
\end{equation*}
denotes the Levi component, where $\mathfrak{m}_{n_{i}} \in \mathrm{GL}_{n_{i}}(\mathbb{R})$. If we take an element $\gamma \in \mathrm{P}_{(n_{1}, \ldots, n_{r})}(\mathbb{R})$ with
\begin{equation*}
\gamma = \begin{pmatrix} I_{n_{1}} & \ast & \cdots & \ast \\ 0 & I_{n_{2}} & \cdots & \ast \\ \vdots & \vdots & \ddots & \vdots \\ 0 & 0 & \cdots & I_{n_{r}} \end{pmatrix} \cdot \begin{pmatrix} \mathfrak{m}_{n_{1}}(\gamma) & 0 & \cdots & 0 \\ 0 & \mathfrak{m}_{n_{2}}(\gamma) & \cdots & 0 \\ \vdots & \vdots & \ddots & \vdots \\ 0 & 0 & \cdots & \mathfrak{m}_{n_{r}}(\gamma) \end{pmatrix},
\end{equation*}
then the Langlands decomposition defines the projection map $\pi_{n_{i}}: \mathrm{P}_{(n_{1}, \ldots, n_{r})}(\mathbb{R}) \to \mathrm{GL}_{n_{i}}(\mathbb{R})$, namely
\begin{equation*}
\gamma \to \mathfrak{m}_{n_{i}}(\gamma).
\end{equation*}

\subsection{Eisenstein series}
In the ensuing discussion, for any real Lie group $\mathrm{G}$ of $n \times n$ matrices, we define $\mathrm{G}(\mathbb{Z}) \coloneqq \mathrm{G}(\mathbb{R}) \cap \mathrm{GL}_{n}(\mathbb{Z})$. The spectral decomposition on $\mathbb{X}_{n}$ reads
\begin{equation}\label{eq:spectral-decomposition}
L^{2}(\mathbb{X}_{n}) = L_{\mathrm{cusp}}^{2}(\mathbb{X}_{n}) \oplus \bigoplus_{(n_{1}, \ldots, n_{r})} L^{2}_{(n_{1}, \ldots, n_{r})}(\mathbb{X}_{n}),
\end{equation}
where the decomposition of the continuous spectrum is indexed by representatives of classes of associate partitions. The cuspidal and continuous components are respectively spanned~by $\mathrm{GL}_{n}$ Hecke--Maa{\ss} cusp forms and Langlands parabolic Eisenstein series.

For each $z \in \mathbb{H}^{n}$ as in~\eqref{eq:Iwasawa} and for every $s = (s_{1}, \ldots, s_{r}) \in \mathbb{C}^{r}$ satisfying the condition
\begin{equation}\label{eq:condition}
\sum_{i = 1}^{r} n_{i} s_{i} = 0,
\end{equation}
we define
\begin{equation*}
I_{s}(z, \mathrm{P}_{(n_{1}, \ldots, n_{r})}) \coloneqq \left(\prod_{j_{1} = n-n_{1}+1}^{n} Y_{j_{1}} \right)^{s_{1}} \left(\prod_{j_{2} = n-n_{1}-n_{2}+1}^{n-n_{1}} Y_{j_{2}} \right)^{s_{2}} \cdots \left(\prod_{j_{r} = 1}^{n_{r}} Y_{j_{r}} \right)^{s_{r}},
\end{equation*}
where $Y_{i} \coloneqq y_{1} y_{2} \cdots y_{i-1}$. 
\begin{definition}
Let $\varphi_{i}$ be a collection of Hecke--Maa{\ss} cusp forms on $\mathrm{GL}_{n_{i}}(\mathbb{R})$ for $1 \leq i \leq r$. Let $s = (s_{1}, \ldots, s_{r})$ satisfy~\eqref{eq:condition}. Then we define the parabolic Eisenstein series by
\begin{equation*}
E_{(n_{1}, \ldots, n_{r})}(z, s, \varphi_{1}, \ldots, \varphi_{r}) \coloneqq \sum_{\gamma \in \mathrm{P}_{(n_{1}, \ldots, n_{r})}(\mathbb{Z}) \backslash \mathrm{SL}_{n}(\mathbb{Z})} \prod_{i = 1}^{r} \varphi_{i}(\mathfrak{m}_{n_{i}}(\gamma z)) I_{s}(\gamma z, \mathrm{P}_{(n_{1}, \ldots, n_{r})}).
\end{equation*}
\end{definition}

From the linear relation~\eqref{eq:condition}, one may eliminate $s_{r}$; hence, $E_{(n_{1}, \ldots, n_{r})}(z, s, \varphi_{1}, \ldots, \varphi_{r})$ may be thought of as a function of $z$ and $s = (s_{1}, \ldots, s_{r-1}) \in \mathbb{C}^{r-1}$. This convention is employed throughout the paper. As an emblematic example, the minimal parabolic Eisenstein series arises from the partition $n = 1+\cdots+1$ and is defined by
\begin{equation}\label{eq:minimal-parabolic}
E_{(1, \ldots, 1)}(z, s) \coloneqq \sum_{\gamma \in \mathrm{P}_{(1, \ldots, 1)}(\mathbb{Z}) \backslash \mathrm{SL}_{n}(\mathbb{Z})} I_{s} (\gamma z).
\end{equation}
Given the constant eigenfunction $\varphi_{0} = \mathbbm{1} \in \mathrm{SL}_{n-1}(\mathbb{Z})$, the partition $n = (n-1)+1$ defines the totally degenerate maximal parabolic Eisenstein series
\begin{equation*}
E_{(n-1, 1)}(z, s, \mathbbm{1}) \coloneqq \sum_{\gamma \in \mathrm{P}_{(n-1, 1)}(\mathbb{Z}) \backslash \mathrm{SL}_{n}(\mathbb{Z})} I_{s} (\gamma z, \mathrm{P}_{(n-1, 1)}).
\end{equation*}
Note that all the Eisenstein series converge absolutely for $\Re(s_{i}) \gg 1$.

\subsection{Incomplete Eisenstein series}
For convergence reasons, it is convenient to address the continuous spectrum via incomplete Eisenstein series rather than the Eisenstein series. If $\eta \in \mathcal{C}_{c}^{\infty}(\mathbb{R}_{+}^{n})$ is a smooth compactly supported function, then we define its $n$-dimensional Mellin transform by
\begin{equation*}
\widetilde{\eta}(s_{1}, \ldots, s_{n}) \coloneqq \int_{0}^{\infty} \cdots \int_{0}^{\infty} \eta (y_{1}, \ldots, y_{n}) y_{1}^{-s_{1}} \cdots y_{n}^{-s_{n}} \frac{dy_{1} \cdots dy_{n}}{y_{1} \cdots y_{n}}.
\end{equation*}
The incomplete Eisenstein series associated to $E_{(n_{1}, \ldots, n_{r})}(z, s, \varphi_{1}, \ldots, \varphi_{r})$ and $\eta$ is now defined by
\begin{multline*}
E_{(n_{1}, \ldots, n_{r})}(z, \eta, \varphi_{1}, \ldots, \varphi_{r}) \coloneqq \frac{1}{(2 \pi i)^{r-1}} \int_{(2)} \cdots \int_{(2)} \widetilde{\eta}(s_{1}, \ldots, s_{r-1})\\
\times E_{(n_{1}, \ldots, n_{r})}(z, s, \varphi_{1}, \ldots, \varphi_{r}) ds_{1} \cdots ds_{r-1},
\end{multline*}
where the properties of $\eta$ ensure the convergence of the integral within an appropriate regime. Consequently, the space $L^{2}_{(n_{1}, \ldots, n_{r})}(\mathbb{X}_{n})$ in~\eqref{eq:spectral-decomposition} is spanned by the incomplete Eisenstein series of type $(n_{1}, \ldots, n_{r})$.

\subsection{Approximation argument}
It is evident from~\cite[Theorem~5.1]{Zhang2019} that the main term in higher rank QUE for degenerate maximal parabolic Eisenstein series emerges from the asymptotic formula
\begin{equation*}
\mu_{n, t}(E_{(1, \ldots, 1)}(\cdot, \eta)) \sim \frac{2 \log t}{\Lambda(n)} \langle E_{(1, \ldots, 1)}(\cdot, \eta), 1 \rangle_{n},
\end{equation*}
where $E_{(1, \ldots, 1)}(z, \eta)$ denotes the minimal parabolic incomplete Eisenstein series corresponding to~\eqref{eq:minimal-parabolic}. For the approximation argument, we fix a smooth compactly supported function $f$. We then construct a finite linear combination $G(z)$ of Hecke--Maa{\ss} cusp forms and incomplete Eisenstein series such that $\norm{f-G}_{\infty} < \epsilon$. The analysis thus boils down to estimating $\mu_{n, t}$ on Hecke--Maa{\ss} cusp forms and incomplete parabolic Eisenstein series. It suffices to address the contribution of the latter since Zhang~\cite[Proposition~4.2]{Zhang2019} established that $\mu_{n, t}(\varphi) = 0$ identically for every Hecke--Maa{\ss} cusp form $\varphi$ on $\mathbb{X}_{n}$.

\section{Watson--Ichino-type formula}
The existence of the Watson--Ichino formula is pivotal when the degree is $2$. For example, the methodologies of Luo and Sarnak~\cite[Equation~(9)]{LuoSarnak2004} and Huang~\cite[Page~1228]{Huang2021-3} rely fundamentally on the triple product identity, which simplifies in the the Eisenstein case to
\begin{equation*}
|\mu_{2, t}(\varphi)|^{2} = \frac{1}{2} \frac{\Lambda(\frac{1}{2}, \varphi)^{2} |\Lambda(\frac{1}{2}+2it, \varphi)|^{2}}{\Lambda(1, \operatorname{ad} \varphi) |\Lambda(1+2it)|^{4}},
\end{equation*}
Here $L(s, \varphi)$ denotes the standard $L$-function attached to $\varphi$, defined for $\Re(s) > 1$ by
\begin{equation*}
L(s, \varphi) \coloneqq \sum_{n = 1}^{\infty} \frac{\lambda_{\varphi}(n)}{n^{s}},
\end{equation*}
where $\lambda_{\varphi}(n)$ denotes the normalised $n$-th Fourier coefficient of $\varphi$. Moreover, let $\rho_{\varphi}(n)$ denote the unnormalised $n$-th Fourier coefficient of $\varphi$ such that $\rho_{\varphi}(n) = \rho_{\varphi}(1) \lambda_{\varphi}(n)$. For brevity, we focus on the case where the parity of $\varphi$ is equal to $1$, in which case the completed $L$-function is defined by
\begin{equation}\label{eq:complete-1}
\Lambda(s, \varphi) \coloneqq L_{\infty}(s, \varphi) L(s, \varphi), \qquad L_{\infty}(s, \varphi) \coloneqq \pi^{-s} \Gamma \left(\frac{s+i t_{\varphi}}{2} \right) \Gamma \left(\frac{s-i t_{\varphi}}{2} \right),
\end{equation}
and enjoys the functional equation $\Lambda(s, \varphi) = \Lambda(1-s, \varphi)$.

Similarly, the adjoint square $L$-function attached to $\varphi$ is defined for $\Re(s) > 1$ by
\begin{equation*}
L(s, \operatorname{ad} \varphi) \coloneqq \zeta(2s) \sum_{n = 1}^{\infty} \frac{\lambda_{\varphi}(n^{2})}{n^{s}},
\end{equation*}
and enjoys the functional equation $\Lambda(s, \operatorname{ad} \varphi) = \Lambda(1-s, \operatorname{ad} \varphi)$ with the completed $L$-function
\begin{equation}\label{eq:complete-2}
\Lambda(s, \operatorname{ad} \varphi) \coloneqq L_{\infty}(s, \operatorname{ad} \varphi) L(s, \operatorname{ad} \varphi),
\end{equation}
where
\begin{equation*}
L_{\infty}(s, \operatorname{ad} \varphi) \coloneqq \pi^{-\frac{3s}{2}} \Gamma \left(\frac{s}{2} \right) \Gamma \left(\frac{s}{2}+i t_{\varphi} \right) \Gamma \left(\frac{s}{2}-i t_{\varphi} \right).
\end{equation*}
In particular, the special value at $s = 1$ becomes
\begin{equation}\label{eq:adjoint} 
\Lambda(1, \operatorname{ad} \varphi) = \frac{8}{|\rho_{\varphi}(1)|^{2}}.
\end{equation}
The following result serves as a substitute of the Watson--Ichino formula~\cite{Ichino2008,Watson2002} and is a crucial ingredient in our subsequent argument.
\begin{proposition}\label{prop:Watson-Ichino-substitute}
Let $\varphi$ be an even Hecke--Maa{\ss} cusp form on $\mathrm{SL}_{2}(\mathbb{Z})$, let $\eta \in \mathcal{C}_{c}^{\infty}(\mathbb{R}_{+}^{n-2})$ be a smooth compactly supported function, and let $\mu_{n, t}$ be defined by~\eqref{eq:inner-product}. Then we have for any $n \geq 3$ that
\begin{equation}\label{eq:Watson-Ichino-substitute}
\mu_{n, t}(E_{(2, 1, \ldots, 1)}(\cdot, \eta, \varphi)) = a_{n} \frac{|\Lambda(2-\frac{n}{2}+int)|^{2}}{|\Lambda(\frac{n}{2}+int)|^{2}} \left\langle \varphi, \left|E \left(\cdot, 1-\frac{n}{4}+\frac{int}{2} \right) \right|^{2} \right\rangle_{2},
\end{equation}
where\footnote{The ellipsis hides an explicitly computable quantity extraneous to the present objective.}
\begin{equation}\label{eq:a-n}
a_{n} \coloneqq \frac{(-1)^{n} (n!)^{2}}{72} \widetilde{\eta}(\cdots).
\end{equation}
If $\varphi$ is odd, then the left-hand side of~\eqref{eq:Watson-Ichino-substitute} vanishes identically.
\end{proposition}

\begin{proof}
Recall that Zhang~\cite[Page~5]{Zhang2019} employs the normalised Haar measure given by
\begin{equation*}
\frac{d\mu_{n}(z)}{n \prod_{k = 2}^{n} \Lambda(k)}.
\end{equation*}
For convenience, we introduce the ad hoc notation
\begin{equation*}
\mu_{n, \nu}(f) \coloneqq \langle f, |E_{(n-1, 1)}(\cdot, \nu, \mathbbm{1})|^{2} \rangle_{n} = \int_{\mathbb{X}_{n}} f(z) |E_{(n-1, 1)}(z, \nu, \mathbbm{1})|^{2} d\mu_{n}(z).
\end{equation*}
For $n \geq 4$, the second step in his computation~\cite[Page 29]{Zhang2019} reads (with some adjustments to the notation)
\begin{multline*}
\frac{1}{n \prod_{k = 2}^{n} \Lambda(k)} \mu_{n, \nu}(E_{(2, 1, \ldots, 1)}(\cdot, \eta, \varphi))
 = \frac{1}{(n-1) \prod_{k = 2}^{n-1} \Lambda(k)} \frac{-(n^{2}-n)}{\Lambda(n)} \frac{|\Lambda(n \nu-1)|^{2}}{|\Lambda(n \nu)|^{2}}\\
\times \mu_{\frac{n \nu-1}{n-1}, n-1}(E_{(2, 1, \ldots, 1)}(\cdot, \eta^{\prime}, \varphi)).
\end{multline*}
This expression now boils down to
\begin{equation*}
\mu_{n, \nu}(E_{(2, 1, \ldots, 1)}(\cdot, \eta, \varphi)) = -n^{2} \frac{|\Lambda(n \nu-1)|^{2}}{|\Lambda(n \nu)|^{2}} \mu_{n-1, \frac{n \nu-1}{n-1}}(E_{(2, 1, \ldots, 1)}(\cdot, \eta^{\prime}, \varphi)).
\end{equation*}
Note that the incomplete Eisenstein series on the left-hand side is associated to the partition $2+1+\cdots+1 = n$; hence $\eta$ is a smooth compactly supported function of $n-2$ variables. The incomplete Eisenstein series on the right-hand side is associated to the partition $2+1+\cdots+1=n-1$; hence $\eta^{\prime}$ is a smooth compactly supported function of $n-3$ variables. More precisely, $\eta$ and $\eta^{\prime}$ are related through their Mellin transforms
\begin{equation*}
\widetilde \eta^{\prime}(s_{1}, \ldots, s_{n-3}) = \widetilde{\eta}(s_{1}+c , \ldots, s_{n-3}+c, -(s_{1}+\cdots+s_{n-3}+c)),
\end{equation*}
where $c = \frac{2-\Re(\nu)}{n-1}-1$. Iterating this process $n-2$ times with $\nu = \frac{1}{2}+it$ and noting that the case of $n = 3$ is peculiar, we obtain the desired claim.
\end{proof}

Moreover, the right-hand side of~\eqref{eq:Watson-Ichino-substitute} undergoes further simplification. If we denote
\begin{equation*}
\mathcal{I}(s, \varphi) \coloneqq \int_{\mathbb{X}_{2}} \varphi(z) E \left(z, 1-\frac{n}{4}+\frac{int}{2} \right) E(z, s) d\mu_{2}(z),
\end{equation*}
then the standard unfolding argument leads to
\begin{equation*}
\mathcal{I}(s, \varphi) = \int_{0}^{\infty} \int_{0}^{1} \varphi(z) E \left(z, 1-\frac{n}{4}+\frac{int}{2} \right) y^{s} \, \frac{dx dy}{y^{2}}.
\end{equation*}
The right-hand side vanishes identically for $\varphi$ odd, whereas for $\varphi$ even, we compute
\begin{equation*}
\mathcal{I}(s, \varphi) = \frac{2\rho_{\varphi}(1)}{\Lambda(2-\frac{n}{2}+int)} \sum_{m = 1}^{\infty} \frac{\lambda_{\varphi}(m) \sigma_{\frac{n}{2}-1-int}(m)}{m^{s-\frac{1}{2}+\frac{n}{4}-\frac{int}{2}}} \int_{0}^{\infty} K_{it_{\varphi}}(2\pi y) K_{\frac{1}{2}-\frac{n}{4}+\frac{int}{2}}(2 \pi y) y^{s} \frac{dy}{y}.
\end{equation*}
Building on Luo and Sarnak~\cite[Section~2]{LuoSarnak1995} and substituting~\cite[Lemma~4.1]{Zhang2019}, it follows that
\begin{align*}
\mathcal{I}(s, \varphi) &= \frac{\rho_{\varphi}(1)}{4\pi^{2s}} \Gamma \left(\frac{s+it_{\varphi}}{2}+\frac{1}{4}-\frac{n}{8}+\frac{int}{4} \right) \Gamma \left(\frac{s-it_{\varphi}}{2}+\frac{1}{4}-\frac{n}{8}+\frac{int}{4} \right)\\
& \quad \times \Gamma \left(\frac{s+it_{\varphi}}{2}-\frac{1}{4}+\frac{n}{8}-\frac{int}{4} \right) \Gamma \left(\frac{s-it_{\varphi}}{2}-\frac{1}{4}+\frac{n}{8}-\frac{int}{4} \right)\\
& \quad \times \frac{L(s-\frac{1}{2}+\frac{n}{4}-\frac{int}{2}, \varphi) L(s+\frac{1}{2}-\frac{n}{4}+\frac{int}{2}, \varphi)}{\Lambda(2-\frac{n}{2}+int) \Lambda(2s)}.
\end{align*}
Setting $s = 1-\frac{n}{4}-\frac{int}{2}$ now yields
\begin{equation*}
\mu_{n, t}(E_{(2, 1, \ldots, 1)}(\cdot, \eta, \varphi)) = a_{n} \frac{|\Lambda(2-\frac{n}{2}+int)|^{2}}{|\Lambda(\frac{n}{2}+int)|^{2}} 
 \mathcal{I}_{\varphi} \left(1-\frac{n}{4}-\frac{int}{2} \right),
\end{equation*}
which is equal to
\begin{multline}\label{eq:full}
\frac{a_{n} \rho_{\varphi}(1)}{4\pi^{2-\frac{n}{2}-int}} \left|\Gamma \left(\frac{3-n}{4}+\frac{it_{\varphi}}{2} \right) \right|^{2} \Gamma \left(\frac{1}{4}+\frac{i(t_{\varphi}-nt)}{2} \right) \Gamma \left(\frac{1}{4}-\frac{i(t_{\varphi}+nt)}{2} \right)\\
\times \frac{L(\frac{1}{2}-int, \varphi) L(\frac{3-n}{2}, \varphi)}{|\Lambda(\frac{n}{2}+int)|^{2}}.
\end{multline}
This expression is consistent with the computation of Huang~\cite[Equation~(3.1)]{Huang2021-3} when $n = 2$. Immediate consequences include the following Watson--Ichino-type formula.
\begin{corollary}\label{prop:Watson-Ichino}
Let $\varphi$ be an even Hecke--Maa{\ss} cusp form on $\mathrm{SL}_{2}(\mathbb{Z})$, let $\eta \in \mathcal{C}_{c}^{\infty}(\mathbb{R}_{+}^{n-2})$ be a smooth compactly supported function, and let $\mu_{n, t}$ be defined by~\eqref{eq:inner-product}. Then we have for any $n \geq 3$ that
\begin{equation}\label{eq:Watson-Ichino}
\mu_{n, t}(E_{(2, 1, \ldots, 1)}(\cdot, \eta, \varphi)) = \frac{a_{n} \rho_{\varphi}(1)}{4\pi^{-int}} \frac{\Lambda(\frac{1}{2}-int, \varphi) \Lambda(\frac{3-n}{2}, \varphi)}{|\Lambda(\frac{n}{2}+int)|^{2}}.
\end{equation}
If $\varphi$ is odd, then the left-hand side of~\eqref{eq:Watson-Ichino} vanishes identically.
\end{corollary}

\begin{proof}
The claim follows from~\eqref{eq:complete-1},~\eqref{eq:complete-2}, and~\eqref{eq:full}.
\end{proof}

\begin{corollary}\label{cor:Zhang}
Let $\varphi$ be an even Hecke--Maa{\ss} cusp form on $\mathrm{SL}_{2}(\mathbb{Z})$, let $\eta \in \mathcal{C}_{c}^{\infty}(\mathbb{R}_{+}^{n-2})$ be a smooth compactly supported function, and let $\mu_{n, t}$ be defined by~\eqref{eq:inner-product}. Then we have for any $n \geq 3$ that
\begin{equation}\label{eq:Zhang-Watson-Ichino}
|\mu_{n, t}(E_{(2, 1, \ldots, 1)}(\cdot, \eta, \varphi))|^{2} = \frac{|a_{n}|^{2}}{2} \frac{|\Lambda(\frac{1}{2}-int, \varphi)|^{2} \Lambda(\frac{3-n}{2}, \varphi)^{2}}{|\Lambda (\frac{n}{2}+int)|^{4} \Lambda(1, \operatorname{ad} \varphi)}.
\end{equation}
If $\varphi$ is odd, then the left-hand side of~\eqref{eq:Zhang-Watson-Ichino} vanishes identically.
\end{corollary}

\begin{proof}
The claim follows from Corollary~\ref{prop:Watson-Ichino} and~\eqref{eq:adjoint}.
\end{proof}

\section{Proof of Theorem~\ref{thm:main}}
This section is devoted to the proof of Theorem~\ref{thm:main}. Using Stirling's approximation~\cite[Equation~(8.327.1)]{GradshteynRyzhik2007}, the asymptotic behaviour of the gamma factors in~\eqref{eq:full} is determined~as
\begin{equation*}
\frac{\Gamma(\frac{1}{4}+\frac{i(t_{\varphi}-nt)}{2}) \Gamma(\frac{1}{4}-\frac{i(t_{\varphi}+nt)}{2})}{|\Gamma(\frac{n}{4}+\frac{int}{2})|^{2}}\\
 = \left(\frac{2}{nt} \right)^{\frac{n-1}{2}} e^{\frac{i\pi}{4}-int \log(\frac{nt}{2e})} \left(1+O \left(\frac{1}{t} \right) \right).
\end{equation*}
It thus follows from~\eqref{eq:full} that
\begin{multline*}
\int_{T}^{2T} \mu_{n, t}(E_{(2, 1, \ldots, 1)}(\cdot, \eta, \varphi)) dt = \frac{a_{n} \rho_{\varphi}(1)}{4\pi^{2-n-int}} \left(\frac{2}{n} \right)^{\frac{n-1}{2}} e^{\frac{i\pi}{4}} \left| \Gamma \left(\frac{3-n}{4}+\frac{it_{\varphi}}{2} \right) \right|^{2} L \left(\frac{3-n}{2}, \varphi \right)\\
\times \int_{T}^{2T} e^{-int \log(\frac{nt}{2e\pi})} \frac{L(\frac{1}{2}-int, \varphi)}{|\zeta(\frac{n}{2}+int)|^{2}} \frac{dt}{t^{\frac{n-1}{2}}}+O_{n} \left(\frac{\sqrt{\log T}}{T^{\frac{n+1}{2}}} \right),
\end{multline*}
where the error term arises from the application of the Cauchy--Schwarz inequality and the result of Jutila~\cite[Theorem~3]{Jutila1996}. It now suffices to estimate the oscillatory first moment
\begin{equation*}
\mathcal{F}(T) \coloneqq \int_{T}^{2T} e^{-int \log(\frac{nt}{2e\pi})} \frac{L(\frac{1}{2}-int, \varphi)}{|\zeta(\frac{n}{2}+int)|^{2}} \frac{dt}{t^{\frac{n-1}{2}}}.
\end{equation*}
The computation of the mean value becomes more straightforward in comparison to the case of $n = 2$ due to the presence of the Riemann zeta function $\zeta(s)$ on the line $\Re(s) = \frac{n}{2} > 1$, which lies within the region of absolute convergence unlike when $n = 2$.

Let $U \in (0, \frac{1}{2})$ be an auxiliary parameter whose optimal size will be determined later. We introduce a smooth weight function $v$ satisfying the following three properties:
\begin{itemize}
\item $v$ is compactly supported on $[1-U, 2+U]$;
\item $v(t) = 1$ for $t\in [1, 2]$ and $0 \leq v(t) \leq 1$ otherwise;
\item $|v^{(j)}(t)| \ll_{j} U^{-j}$ for all fixed $j \in \mathbb{N}_{0}$ and $t \in \mathbb{R}$.
\end{itemize}
In the ensuing analysis, let $\epsilon > 0$ represent a small quantity to be fixed later. We define
\begin{equation}\label{eq:smooth}
\mathcal{F}_{v}(T) \coloneqq \int_{0}^{\infty} e^{-int \log(\frac{nt}{2e\pi})} \frac{L(\frac{1}{2}-int, \varphi)}{|\zeta(\frac{n}{2}+int)|^{2}} v \left(\frac{t}{T} \right) \frac{dt}{t^{\frac{n-1}{2}}}.
\end{equation}
Applying the Cauchy--Schwarz inequality in conjunction with~\cite[Theorem~3]{Jutila1996} once again implies
\begin{align*}
|\mathcal{F}(T)-\mathcal{F}_{v}(T)| &\ll_{n, \varphi} ([t\log t+t]^{T(1+U)}_{T}+T^{\frac{2}{3}+\epsilon})^{\frac{1}{2}} \sqrt{\frac{U}{T^{n-2}}} \\
&\ll_{n, \varphi} T^{-\frac{n-3}{2}} U \sqrt{\log T}+T^{\frac{4}{3}-\frac{n}{2}+\epsilon} \sqrt{U}.
\end{align*}
The Riemann zeta function in the denominator in~\eqref{eq:smooth} can be expanded as
\begin{equation*}
\frac{1}{|\zeta(\frac{n}{2}+int)|^{2}} = \sum_{1 \leq k, \ell \leq K} \frac{\mu(k) \mu(\ell)}{(k \ell)^{\frac{n}{2}}} \left(\frac{k}{\ell} \right)^{int}+O(K^{1-\frac{n}{2}})
\end{equation*}
for a parameter $K > 0$ at our discretion. The error term contributes $O(K^{1-\frac{n}{2}} T^{-\frac{n-3}{2}} \sqrt{\log T})$ in~\eqref{eq:smooth}. For $t \asymp T$, the approximate functional equation~\cite[Theorem~5.3]{IwaniecKowalski2004} reads
\begin{equation}\label{eq:approx}
L \left(\frac{1}{2}-int, \varphi \right) = \sum_{m \leq T^{1+\epsilon}} \frac{\lambda_{\varphi}(m)}{m^{\frac{1}{2}-int}} \mathcal{W}_{t}^{-}(m)-i \left(\frac{nt}{2e\pi} \right)^{2int} \sum_{m \leq T^{1+\epsilon}} \frac{\lambda_{\varphi}(m)}{m^{\frac{1}{2}+int}} \mathcal{W}_{t}^{+}(m)+O_{n, \varphi}(T^{-\frac{1}{2}+\epsilon}),
\end{equation}
where
\begin{equation*}
\mathcal{W}_{t}^{\pm}(y) = \frac{1}{2\pi i} \int_{\epsilon-iT^{\epsilon}}^{\epsilon+iT^{\epsilon}} \left(\frac{nt}{2\pi y} \right)^{s} e^{s^{2} \pm \frac{i\pi s}{2}} \frac{ds}{s}.
\end{equation*}
The error term contributes $O_{n, \varphi}(T^{-\frac{n-1}{2}+\epsilon})$ in~\eqref{eq:smooth}, and it remains to estimate the expression
\begin{equation}\label{eq:after-AFE}
\frac{1}{2\pi i} \int_{\epsilon-iT^{\epsilon}}^{\epsilon+iT^{\epsilon}} \sum_{m \leq T^{1+\epsilon}} \sum_{1 \leq k, \ell \leq K} \frac{\lambda_{\varphi}(m) \mu(k) \mu(\ell)}{\sqrt m (k \ell)^{\frac{n}{2}}} \int_{0}^{\infty} e^{ih(t)} v_{1}(t) dt \left(\frac{n}{2\pi m} \right)^{s} e^{s^{2}-\frac{i\pi s}{2}} \frac{ds}{s},
\end{equation}
where $v_{1}(t) \coloneqq t^{s-\frac{n-1}{2}} v(\frac{t}{T})$ and $h(t) \coloneqq -nt \log(\frac{knt}{2e\pi \ell m})$. It suffices to concentrate on the first term in~\eqref{eq:approx}, since the treatment of the second term proceeds similarly. We are now~prepared to implement the stationary phase analysis to the oscillatory integral
\begin{equation*}
\mathcal{J} \coloneqq \int_{0}^{\infty} e^{ih(t)} v_{1}(t) dt.
\end{equation*}
We observe $v_{1}$ is a smooth compactly supported function on $[T(1-U), T(2+U)]$ satisfying
\begin{equation*}
v_{1}^{(j)}(t) \ll_{j} T^{-\frac{n-1}{2}+\epsilon}(TU)^{-j}
\end{equation*}
for all fixed $j \in \mathbb{N}_{0}$ and $U \leq T^{-\epsilon}$. Moreover, $h$ is a smooth function such that $|h^{(j)}(t)| \asymp_{j} T^{1-j}$ for $j \geq 2$ and $h^{\prime}(t) = 0$ if and only if $t = \frac{2\pi \ell m}{kn}$. In anticipation of further simplifications, it is convenient to handle the following~two cases separately.

\medskip

\noindent \textbf{Case I.} $\frac{2\pi \ell m}{kn} < \frac{T}{2}$ or $ \frac{2\pi \ell m}{kn} > 3T$.

\medskip

In this case, we have $|h^{\prime}(t)| \gg 1$, and~\cite[Lemma~8.1]{BlomerKhanYoung2013} yields
\begin{equation*}
\mathcal{J} \ll_{A} T^{-\frac{n-3}{2}+\epsilon}(T^{-\frac{A}{2}}+(TU)^{-A}) \ll_{A} T^{-\frac{A}{2}-\frac{n-3}{2}+\epsilon}
\end{equation*}
for any sufficiently large constant $A > 0$ and $U \geq T^{-\frac{1}{2}}$. The rapid decay of $s \mapsto e^{s^{2}-\frac{i\pi s}{2}}$ on vertical lines justifies that the contribution of such terms is negligibly small.

\medskip
 
\noindent \textbf{Case II.} $\frac{T}{2} \leq \frac{2\pi \ell m}{kn} \leq 3T$.

\medskip
 
In this case, we have $|h^{\prime}(t)| \ll 1$ and $|h^{\prime \prime}(t)| \gg T^{-1}$, and~\cite[Proposition 8.2]{BlomerKhanYoung2013}~with $U \geq T^{-\frac{1}{2}+\delta}$ for fixed $0 < \delta < \frac{1}{10}$ yields
\begin{equation*}
\mathcal{J} = \frac{e(\frac{\ell m}{k})}{\sqrt{\frac{kn}{2\pi \ell m}}} w_{0} \left(\frac{2\pi \ell m}{kn} \right)+O_{A}(T^{-A}),
\end{equation*}
where $w_{0}$ is a smooth compactly supported function on $[T(1-U), T(2+U)]$ satisfying
\begin{equation*}
|w_{0}^{(j)}(t)| \ll_{j} T^{-\frac{n-1}{2}+\epsilon}(TU)^{-j}.
\end{equation*}
We choose $K$ such that $ \frac{2\pi \ell m}{kn} > 3T$ for $m \leq T^{1+\epsilon}$, ensuring that $0 < K \leq \frac{2\pi}{3n} T^{\epsilon}$ suffices. As a result, the expression~\eqref{eq:after-AFE} is equal up to an absolute and effectively computable constant~to
\begin{equation*}
\int_{\epsilon-iT^{\epsilon}}^{\epsilon+iT^{\epsilon}} \sum_{1 \leq k, \ell \leq K} \frac{\mu(k) \mu(\ell)}{k^{\frac{n+1}{2}}\ell^{\frac{n-1}{2}}} D \left(s, \varphi; \frac{\ell}{k} \right) \left(\frac{n}{2\pi} \right)^{s} e^{s^{2}-\frac{i\pi s}{2}} \frac{ds}{s}+O_{A}(T^{-A}),
\end{equation*}
where
\begin{equation*}
D \left(s, \varphi; \frac{\ell}{k} \right) \coloneqq \sum_{m = 1}^{\infty} \frac{\lambda_{\varphi}(m) e(\frac{\ell m}{k})}{m^{s}} w \left(\frac{\ell m}{k} \right)
\end{equation*}
upon setting $w(x) \coloneqq w_{0}(\frac{2\pi x}{n})$. Moreover, we define the cuspidal analogue of the Estermann zeta function by
\begin{equation*}
L\left(s, \varphi; \frac{\ell}{k} \right) \coloneqq \sum_{m = 1}^{\infty} \frac{\lambda_{\varphi}(m) e(\frac{\ell m}{k})}{m^{s}}
\end{equation*}
for $\Re(s) > 1$. After meromorphic continuation, Mellin inversion leads to
\begin{equation}\label{eq:Mellin}
D \left(s, \varphi; \frac{\ell}{k} \right) = \frac{1}{2\pi i} \int_{(2)} L \left(s+s^{\prime}, \varphi; \frac{\ell}{k} \right) \widehat w(s^{\prime}) \left(\frac{k}{\ell} \right)^{s^{\prime}} ds^{\prime},
\end{equation}
where $\widehat{w}(s^{\prime})$ denotes the Mellin transform of $w$ in the variable $s^{\prime}$. Integration by parts implies for any $j \in \mathbb{N}_{0}$ that
\begin{equation*}
\widehat{w}(s^{\prime}) = \frac{(-1)^{j}}{s^{\prime}(s^{\prime}+1) \cdots (s^{\prime}+j-1)} \int_{0}^{\infty} w^{(j)}(x) x^{s^{\prime}+j-1} dx \ll_{j} \frac{T^{\sigma-\frac{n-1}{2}+\epsilon} U^{-j}}{1+|\tau^{\prime}|^{j}},
\end{equation*}
where $s^{\prime} \coloneqq \sigma+i\tau^{\prime}$. If we reduce the fraction to lowest terms $ \frac{\ell}{k} = \frac{u}{q}$ with $(u, q) = 1$, then it follows from the functional equation and the Phragm\'{e}n--Lindel\"{o}f principle that
\begin{equation*}
L\left(\frac{1}{2}+i\tau^{\prime}, \varphi; \frac{u}{q} \right) \ll_{\varphi} (q(1+|\tau^{\prime}|))^{\frac{1+\epsilon}{2}}.
\end{equation*}
Since $q \leq k$, shifting the contour in~\eqref{eq:Mellin} to $\Re(s^{\prime}) = \frac{1}{2}-\epsilon$ yields
\begin{align*}
D \left(s, \varphi; \frac{\ell}{k} \right) &\ll_{\varphi} k^{\frac{1+\epsilon}{2}} T^{1-\frac{n}{2}} U^{-2} \left(\frac{k}{\ell} \right)^{\frac{1}{2}-\epsilon} \int_{\mathbb{R}} \frac{(1+|\tau+\tau^{\prime}|)^{\frac{1+\epsilon}{2}}}{1+|\tau^{\prime}|^{2}} d\tau^{\prime}\\
&\ll_{\varphi} k^{1-\frac{\epsilon}{2}} \ell^{-\frac{1}{2}+\epsilon} T^{1-\frac{n}{2}} U^{-2} P(\tau),
\end{align*}
where $P$ is a fixed well-behaved function of polynomial growth with respect to $\tau \coloneqq \Im(s)$.

Synthesising the preceding observations, we obtain
\begin{multline*}
\int_{T}^{2T} \mu_{n, t}(E_{(2, 1, \ldots, 1)}(\cdot, \eta, \varphi)) dt\\
\ll_{n, \varphi} T^{-\frac{n-3}{2}} U \sqrt{\log T}+T^{\frac{4}{3}-\frac{n}{2}+\epsilon} \sqrt{U}+K^{1-\frac{n}{2}} T^{-\frac{n-3}{2}} \sqrt{\log T}+T^{1-\frac{n}{2}} U^{-2}.
\end{multline*}
It is now convenient to set $K = \frac{2\pi}{3n} T^{\epsilon}$ and $U = T^{-\eta}$ for $\epsilon \leq \eta < \frac{1}{2}$. Under these optimisations, the right-hand side is bounded by
\begin{equation*}
(T^{-\eta}+T^{-\frac{1}{6}-\frac{\eta}{2}+\epsilon}+T^{-\frac{\epsilon n}{2}+\epsilon}+T^{-\frac{1}{2}+2\eta}) T^{-\frac{n-3}{2}}
 \sqrt{\log T},
\end{equation*}
which achieves a minimum when
\begin{equation*}
\eta = \frac{2}{15}+\frac{2\epsilon}{5}, \qquad \epsilon = \frac{7}{3(5n-2)}.
\end{equation*}
Altogether, we conclude that
\begin{equation*}
\int_{T}^{2T} \mu_{n, t}(E_{(2, 1, \ldots, 1)}(\cdot, \eta, \varphi)) dt \ll_{n, \varphi} T^{-\frac{n-3}{2}-\frac{7(n-2)}{6(5n-2)}} \sqrt{\log T}.
\end{equation*}
The proof of Theorem~\ref{thm:main} is complete.

\section{Weighted quantum variance}
This section offers an asymptotic formula for a weighted version of the quantum~variance $Q_{n}(\varphi, \psi)$. We denote by $\mathcal{G}_{n, \varphi}(t)$ the gamma factors in the expression for $\mu_{n, t}(E_{(2, 1, \ldots, 1)}(\cdot, \eta, \varphi))$. Applying Stirling's approximation for $t \asymp T$ and $t_{\varphi}$ fixed leads to
\begin{equation}\label{eq:gamma}
|\mathcal{G}_{n, \varphi}(t)|^{2} = \frac{|\Gamma(\frac{1}{4}+\frac{i(t_{\varphi}-nt)}{2}) \Gamma(\frac{1}{4}-\frac{i(t_{\varphi}+nt)}{2})|^{2}}{|\Gamma(\frac{n}{4}+\frac{int}{2})|^{4}} = \left(\frac{2}{n t} \right)^{n-1}+O_{n}(T^{-n}).
\end{equation}
The following proposition serves as an $n$-dimensional analogue of~\cite[Theorem~5]{Huang2021-3}.
\begin{proposition}\label{prop:quantum-variance}
Keep the notation as above. Then we have for $\varphi = \psi$ even that
\begin{multline*}
\lim_{T \to \infty} \frac{T^{n-2}}{\log T} \int_{T}^{2T} \left|\zeta \left(\frac{n}{2}+int \right) \right|^{4} |\mu_{n, t}(E_{2, 1, \ldots, 1}(\cdot, \eta, \varphi))|^{2} dt\\
 = \frac{6 \log n}{\pi^{2}} |a_{n}|^{2} \left(\frac{2\pi}{n} \right)^{n-1} \cosh(\pi t_{\varphi}) \Lambda \left(\frac{3-n}{2}, \varphi \right)^{2},
\end{multline*}
where $a_{n}$ is defined by~\eqref{eq:a-n}.
\end{proposition}

\begin{proof}
Using Corollary~\ref{cor:Zhang}, we rearrange
\begin{multline*}
\int_{T}^{2T} \left|\zeta \left(\frac{n}{2}+int \right) \right|^{4} \left|\mu_{n, t}(E_{2, 1, \ldots, 1}(\cdot, \eta, \varphi)) \right|^{2} dt \\
= \frac{|a_{n}|^{2} \pi^{n-1}}{2} \frac{\Lambda(\frac{3-n}{2}, \varphi)^{2}}{\Lambda(1, \operatorname{ad} \varphi)} 
\int_{T}^{2T} |\mathcal{G}_{n, \varphi}(t)|^{2} \left|L \left(\frac{1}{2}-int, \varphi \right) \right|^{2} dt.
\end{multline*}
It follows from~\eqref{eq:gamma} that the integral on the right-hand side evaluates to
\begin{equation*}
\left(\frac{2}{n} \right)^{n-1} \int_{T}^{2T} \left|L \left(\frac{1}{2}-int, \varphi \right) \right|^{2} \frac{dt}{t^{n-1}}+O_{n, \varphi} \left(T^{-n} \int_{T}^{2T} \left|L \left(\frac{1}{2}-int, \varphi \right) \right|^{2} dt \right).
\end{equation*}
Recall the second moment estimate in long intervals due to Jutila~\cite[Theorem~3]{Jutila1996}, namely
\begin{equation*}
\int_{T}^{2T} \left|L \left(\frac{1}{2}+it, \varphi \right) \right|^{2} dt = \frac{12}{\pi^{2}} \Lambda(1, \operatorname{ad} \varphi) \cosh(\pi t_{\varphi}) T(\log T+B_{\varphi})+O_{\varphi, \epsilon}(T^{\frac{2}{3}+\epsilon}),
\end{equation*}
where $B_{\varphi}$ is an explicitly computable constant dependent on $\varphi$. In accordance with~\cite[Section~3]{Huang2021-3}, we derive
\begin{multline*}
\int_{T}^{2T} |\mathcal{G}_{n, \varphi}(t)|^{2} \left|L \left(\frac{1}{2}-int, \varphi \right) \right|^{2} dt\\
 = \frac{12 \log n}{\pi^{2}} \left(\frac{2}{n} \right)^{n-1} \Lambda(1, \operatorname{ad} \varphi) \cosh(\pi t_{\varphi}) \frac{\log T}{T^{n-2}}+O_{n, \varphi}(T^{2-n}).
\end{multline*}
Taking the limit $T \to \infty$ completes the proof of Proposition~\ref{prop:quantum-variance}.
\end{proof}

The statement of Proposition~\ref{prop:quantum-variance} is recast as
\begin{multline*}
\lim_{T \to \infty} \frac{T^{n-2}}{\log T} \int_{T}^{2T} \left|\zeta \left(\frac{n}{2}+int \right) \right|^{4} 
|\mu_{n, t}(E_{2, 1, \ldots, 1}(\cdot, \eta, \varphi))|^{2} dt\\
 = \frac{12 |a_{n}|^{2} \log n}{\pi} \left(\frac{2\pi}{n} \right)^{n-1} L \left(\frac{3-n}{2}, \varphi \right)^{2} V_{n}(\varphi),
\end{multline*}
where $V_{n}(\varphi)$ is defined by~\eqref{eq:V}.

\section{Proof of Theorem~\ref{thm:quantum-variance}} 
Applying Corollary~\ref{prop:Watson-Ichino} for any two even Hecke--Maa{\ss} cusp forms $\varphi$ and $\psi$, we deduce
\begin{multline*}
\int_{T}^{2T} \mu_{n, t}(E_{2, 1, \ldots, 1}(\cdot, \eta, \varphi)) \overline{\mu_{n, t}(E_{2, 1, \ldots, 1}(\cdot, \eta, \psi))} dt\\
 = \frac{|a_{n}|^{2}}{16\pi^{4-2n}} \rho_{\varphi}(1) \overline{\rho_{\psi}(1)} \left|\Gamma \left(\frac{3-n}{4}+\frac{it_{\varphi}}{2} \right) \right|^{2} \left|\Gamma \left(\frac{3-n}{4}+\frac{it_{\psi}}{2} \right) \right|^{2} L \left(\frac{3-n}{2}, \varphi \right) L \left(\frac{3-n}{2}, \psi \right)\\ \times \int_{T}^{2T} \mathcal{G}_{n, \varphi}(t) \overline{G_{n, \psi}(t)} \frac{L(\frac{1}{2}-int, \varphi) L(\frac{1}{2}+int, \psi)}{|\zeta(\frac{n}{2}+int)|^{4}}.
\end{multline*}
Stirling's approximation implies that the right-hand side equals
\begin{multline*}
\frac{|a_{n}|^{2}}{16\pi^{4-2n}} \left(\frac{2}{n} \right)^{n-1} \rho_{\varphi}(1) \overline{\rho_{\psi}(1)} \left|\Gamma \left(\frac{3-n}{4}+\frac{it_{\varphi}}{2} \right) \right|^{2} \left|\Gamma \left(\frac{3-n}{4}+\frac{it_{\psi}}{2} \right) \right|^{2}\\
\times L \left(\frac{3-n}{2}, \varphi \right) L \left(\frac{3-n}{2}, \psi \right) \int_{T}^{2T} \frac{L(\frac{1}{2}-int, \varphi) L(\frac{1}{2}+int, \psi)}{|\zeta(\frac{n}{2}+int)|^{4}} \frac{dt}{t^{n-1}}+O_{n, \varphi, \psi, \epsilon}(T^{1-n+\epsilon}).
\end{multline*}
The integration over $t$ closely resembles that of Huang~\cite[Equation~(4.2)]{Huang2021-3}. Following~the procedure in~\cite[Section~4]{Huang2021-3} \textit{mutatis mutandis}, we obtain the expression
\begin{equation*}
\frac{1}{T} \int_{T}^{2T} \frac{L(\frac{1}{2}-int, \varphi) L(\frac{1}{2}+int, \psi)}{|\zeta(\frac{n}{2}+int)|^{4}} \frac{dt}{t^{n-1}}
 = \mathcal{D}_{\varphi, \psi}+\mathcal{D}_{\psi, \varphi}+\mathcal{O}_{\varphi, \psi}+\mathcal{O}_{\psi, \varphi}+O_{n, \varphi, \psi, \epsilon}(T^{-n+\epsilon}),
\end{equation*}
where the notation is borrowed from Huang with the only difference being that $2$ is replaced by $n \geq 3$. The diagonal contribution is explicitly computed in terms of certain well-behaved Dirichlet series, whereas the off-diagonal contribution is estimated via the $\delta$-symbol method. Consequently, the diagonal term dominates the off-diagonal term, leading to the asymptotic formula\footnote{There exists a typographical error in the argument of the gamma function in the first display on~\cite[Page 1240]{Huang2021-3}.}
\begin{multline*}
\frac{T^{n-2}}{\log T} \int_{T}^{2T} |\mu_{n, t}(E_{2, 1, \ldots, 1}(\cdot, \eta, \varphi))|^{2} dt \sim \frac{|a_{n}|^{2}}{16\pi^{4-2n}} \left(\frac{2}{n} \right)^{n-1} |\rho_{\varphi}(1)|^{2} \left|\Gamma \left(\frac{3-n}{4}+\frac{it_{\varphi}}{2} \right) \right|^{4}\\ \times L \left(\frac{3-n}{2}, \varphi \right)^{2} \frac{2 \log n}{\zeta(2)} L(1, \operatorname{ad} \varphi) H_{\varphi},
\end{multline*}
where $H_{\varphi} \coloneqq H_{\varphi, \varphi}(0, \frac{n}{2}-1, \frac{n}{2}-1)$ for $H_{\varphi, \psi}(s, s_{1}, s_{2})$ as defined by~\cite[Equation~(5.3)]{Huang2021-3}.~If we assemble the resulting factors into
\begin{equation}\label{eq:C-n}
C_{n}(\varphi) \coloneqq \frac{12 |a_{n}|^{2} \log n}{\pi} \left(\frac{2\pi}{n} \right)^{n-1} H_{\varphi},
\end{equation}
then it follows from~\eqref{eq:V} and~\eqref{eq:adjoint} that
\begin{equation*}
\frac{T^{n-2}}{\log T} \int_{T}^{2T} |\mu_{n, t}(E_{2, 1, \ldots, 1}(\cdot, \eta, \varphi))|^{2} dt \sim C_{n}(\varphi) L \left(\frac{3-n}{2}, \varphi \right)^{2} V_{n}(\varphi).
\end{equation*}
Furthermore, if $\varphi \ne \psi$, then
\begin{multline*}
T^{n-2} \int_{T}^{2T} \mu_{n, t}(E_{2, 1, \ldots, 1}(\cdot, \eta, \varphi)) \overline{\mu_{n, t}(E_{2, 1, \ldots, 1}(\cdot, \eta, \psi))} dt\\ \sim \frac{3|a_{n}|^{2} \log n}{4\pi^{6-2n}} \left(\frac{2}{n} \right)^{n-1} \rho_{\varphi}(1) \overline{\rho_{\psi}(1)} \left|\Gamma \left(\frac{3-n}{4}+\frac{it_{\varphi}}{2} \right) \right|^{2} \left|\Gamma \left(\frac{3-n}{4}+\frac{it_{\psi}}{2} \right) \right|^{2}\\
\times L \left(\frac{3-n}{2}, \varphi \right) L \left(\frac{3-n}{2}, \psi \right) L(1, \varphi \otimes \psi) H_{\varphi, \psi}.
\end{multline*}
Taking the limit $T \to \infty$ completes the proof of Theorem~\ref{thm:quantum-variance}.


\newcommand{\etalchar}[1]{$^{#1}$}
\providecommand{\bysame}{\leavevmode\hbox to3em{\hrulefill}\thinspace}
\providecommand{\MR}{\relax\ifhmode\unskip\space\fi MR }
\providecommand{\MRhref}[2]{%
  \href{http://www.ams.org/mathscinet-getitem?mr=#1}{#2}
}
\providecommand{\Zbl}{\relax\ifhmode\unskip\space\thinspace\fi Zbl }
\providecommand{\MR}[2]{%
  \href{https://zbmath.org/?q=an:#1}{#2}
}
\providecommand{\doi}{\relax\ifhmode\unskip\space\thinspace\fi DOI }
\providecommand{\MR}[2]{%
  \href{https://doi.org/#1}{#2}
}
\providecommand{\SSNI}{\relax\ifhmode\unskip\space\thinspace\fi ISSN }
\providecommand{\MR}[2]{%
  \href{#1}{#2}
}
\providecommand{\ISBN}{\relax\ifhmode\unskip\space\thinspace\fi ISBN }
\providecommand{\MR}[2]{%
  \href{#1}{#2}
}
\providecommand{\arXiv}{\relax\ifhmode\unskip\space\thinspace\fi arXiv }
\providecommand{\MR}[2]{%
  \href{#1}{#2}
}
\providecommand{\href}[2]{#2}

\end{document}